\documentclass{article}
\usepackage{amsmath, amsfonts, amsthm, amssymb}
\usepackage{graphicx}
\usepackage{float}
\usepackage{verbatim}% for comment
\usepackage[usenames]{color}
\usepackage[usenames,dvipsnames]{xcolor}

\numberwithin{equation}{section}
\allowdisplaybreaks

\newtheorem{thm}{Theorem}[section]
\newtheorem{lma}[thm]{Lemma}
\newtheorem{cor}[thm]{Corollary}

\newtheorem{prop}[thm]{Proposition}

\theoremstyle{definition}

\newtheorem{rem}[thm]{Remark}

\renewcommand{\geq}{\geqslant}
\renewcommand{\leq}{\leqslant}

\title{First and second moments for self-similar couplings and Wasserstein distances}

\author{Jonathan M. Fraser\\ \\
\emph{Mathematics Institute, Zeeman Building,}\\ \emph{University of Warwick, Coventry, CV4 7AL, UK}\\ \emph{e-mail: jon.fraser32@gmail.com}}

\begin{document}
\maketitle

\begin{abstract}
We study aspects of the Wasserstein distance in the context of self-similar measures.  Computing this distance between two measures involves minimising certain moment integrals over the space of \emph{couplings}, which are measures on the product space with the original measures as prescribed marginals. We focus our attention on self-similar measures associated to equicontractive iterated function systems satisfying the open set condition and consisting of two maps on the unit interval.  We are particularly interested in understanding the restricted family of \emph{self-similar} couplings and our main achievement is the explicit computation of the 1st and 2nd moment integrals for such couplings.  We show that this family is enough to yield an explicit formula for the 1st Wasserstein distance and provide non-trivial upper and lower bounds for the 2nd Wasserstein distance.
\end{abstract}

\section{Introduction}

The Wasserstein metric is widely used as an informative and computable distance function between mass distributions. In computer science it is commonly referred to as the `earth mover's distance' and is a measure of the `work' required to change one distribution into the other.  For discrete distributions on finite sets one can develop efficient algorithms to determine the distance, but in the non-discrete setting calculations can be far from trivial and involve minimising certain moment integrals over the space of \emph{couplings}, which are measures on the product space with the original measures as prescribed marginals.
\\ \\
In this paper we study the 1st and 2nd moment integrals for self-similar couplings of pairs of self-similar measures  arising from equicontractive iterated function systems satisfying the open set condition (OSC) and consisting of two maps on the unit interval.  Given two such measures, the family of self-similar couplings is a 1-parameter family and we are able to give an explicit formula for the 1st and 2nd moments for all measures in this family in terms of this parameter and the defining parameters of the original measures.  This gives natural upper bounds on the 1st and 2nd Wasserstein distances between the original measures and leads us to the following natural questions: `Can the Wasserstein distances be realised by self-similar couplings?' and `how do the 1st and 2nd moment integrals depend on the defining parameters?'  In the case of the 1st distance, we use the Kantorovich-Rubinstein duality theorem, which involves maximising the integral of 1-Lipschitz test functions with respect to the difference of the two measures, to prove that self-similar couplings are indeed sufficient. We thus derive an explicit formula for the 1st Wasserstein distance in terms of the different probability vectors, the contraction parameter, and the translation vectors and, moreover, can exhibit an explicit coupling which realises the distance.   Once we have the formula for the 1st Wasserstein distance we are able to make the following peculiar observation.  If the translation vectors are chosen such that the end points of the unit interval are in the support of the measure (i.e. the support is the middle $(1-2c)$ Cantor set), then the 1st Wasserstein distance does not depend on the contraction parameter $c$, but in all other cases it does.  It appears that this phenomenon does not occur for the 2nd Wasserstein distance.  We conclude with a detailed discussion of our assumptions and prospects for future work.  
\\ \\

\subsection{The Wasserstein metric}

Let $(X,d)$ be a compact metric space and $\mathcal{P}(X)$ be the set of Borel probability measures on $X$.  For measures $\mu, \nu \in \mathcal{P}(X)$, let $\Gamma(\mu,\nu)$ be the set of \emph{couplings} of $\mu$ and $\nu$, i.e., the set of Borel probability measures on $X \times X$ with marginals $\mu$ and $\nu$ on the first and second coordinate, respectively.  For $\rho\geq 1$, the $\rho$th \emph{Wasserstein distance} between $\mu$ and $\nu$ is defined by
\[
W_\rho(\mu, \nu) \ = \ \inf_{\gamma \in \Gamma(\mu,\nu)} \Bigg( \int_{X \times X} d(x,y)^\rho  \, \text{d} \gamma(x,y) \Bigg)^{1/\rho}. 
\]
It can be shown that $(\mathcal{P}(X), W_\rho)$ is a separable and complete metric space.  Moreover, convergence in $W_\rho$ is equivalent to weak convergence of measures.  Although it is surplus to our requirements, it is possible to define the Wasserstein metric on non-compact spaces. In this case $\mathcal{P}(X)$ must be restricted to measures with finite $\rho$th moment and convergence in $W_\rho$ is then equivalent to weak convergence plus convergence of the $\rho$th moment.   For more details on the basic properties of $W_\rho$ see \cite[Chapter 6]{villani}.  An important property for our purposes is that $W_{\rho_1} \leq W_{\rho_2}$ for $\rho_1 \leq \rho_2$, which is a simple consequence of H\"older's inequality.  The Wasserstein metric has its origins in \emph{transportation theory} and in particular \emph{optimal transport}.  Consider the following real world problem.  You have a certain fixed number $N$ units of some product stored in another fixed number $n \ll N$ warehouses according to a given configuration.  Due to a shift in supply and demand, you now wish to store the same $N$ units in the same $n$ warehouses, but in a different configuration.  This leaves you with an optimal transportation problem of the form, `What is the least amount of `work' needed to be done to make the required change in configuration?'  If `work' is interpreted as the product of the number of units with the distance travelled, then the answer is the 1st Wasserstein distance.  To fit in to the abstract framework described above, the warehouses are replaced by points in the plane and the units of product are replaced by point masses with weight $1/N$. For more information on the history and development of this concept, see \cite{villani} and the references therein.   As Villani points out \cite[Remark 6.6]{villani}, the 1st and 2nd Wasserstein distances are the most important in theory and in applications, with $W_1$ often simpler to compute, but with $W_2$ often a better reflection of the relevant geometric properties of the system.  The Wasserstein distance has found its way into a myriad of different fields where one wants a computable and meaningful way of gauging the distance between two distributions, measures, large data sets etc.  For a recent example see \cite{lunel}, where it has been effectively used to uncover long term dynamical properties of different systems by measuring the distance between probability distributions derived from time series analysis.  These ideas are applied to large data sets taken from studying physiological and neurological outputs from the human body, with many potential applications in medical science.

\subsection{Our class of self-similar measures and self-similar couplings}

Rather than computing the Wasserstein metric itself, the main goal of this article is to understand the moment integrals for at least a large and representative family of couplings and to develop techniques for handling such integrals. We wish to analyse a class of measures which is both simple to describe but with enough freedom and potential for complexity that the results are illuminating and interesting.  One natural such class is that of \emph{self-similar measures}, which are defined as follows.  Let $(X,d)$ be a compact metric space and $\{S_i\}_{i \in \mathcal{I}}$ be a finite collection of contracting similitudes mapping $X$ into itself.   Such a collection is called an \emph{iterated function system} (IFS). To this IFS, associate a probability vector $\{p_i\}_{i \in \mathcal{I}}$, with each $p_i \in (0,1)$ and such that $\sum_{i \in \mathcal{I}} p_i = 1$.  It follows that there is a unique non-empty compact set $F \subseteq X$ satisfying
\[
F \ = \ \bigcup_{i \in \mathcal{I}} \  S_i(F)
\]
called the \emph{self-similar set} associated with the system and a unique Borel probability measure $\mu \in \mathcal{P}(X)$ satisfying
\[
\mu \ = \ \sum_{i \in \mathcal{I}} \ p_i \, \mu \circ S_i^{-1}
\]
called the \emph{self-similar measure} associated with the system, which is supported on $F$.  For a review of self-similar sets and measures see \cite[Chapters 9 and 17]{falconer} and the references therein. Ideally, we would like to compute the Wasserstein distances between any two different self-similar measures defined via the same IFS and to understand the behaviour of the moment integrals for a natural family of couplings, for example, Bernoulli measures for the product system.  This is likely a very difficult problem in full generality, not least because the self-similar support can have complicated overlaps, and so we significantly simplify the situation.  We believe our methods could be adapted to deal with a more general setting than what follows, but we refrain from making the result as general as possible in order to aid clarity of exposition.  We include a discussion along these lines in Section \ref{discussion}.  Let $X = [0,1]$, $d$ be induced by the Euclidean norm $\lvert \cdot \rvert$, $c \in (0,1/2]$, $t_1 \in [0,1-2c]$ and $t_2 \in [t_1+c, 1-c]$.  Our IFS will consist of the pair of contractions defined by
\[
S_1(x) = cx+t_1 \qquad \text{and} \qquad S_2(x) = cx+t_2.
\]
The resulting self-similar set $F$ is a totally disconnected Cantor like subset of $[0,1]$, unless $c = 1/2$ in which case it is the whole interval.  If $c = 1/3$, $t_1=0$ and $t_2 = 2/3$, then we obtain the classical middle third Cantor set, one of the most famous and earliest examples of a fractal.  For $p \in (0,1)$, let $\mu_p \in \mathcal{P}(F)$ denote the self-similar measure supported on $F$ corresponding to choosing $p_1 = p$ and $p_2 = 1-p$.
\\ \\
The restrictions on $c, t_1$ and $t_2$ guarantee that the resulting IFS satisfies the OSC with the open set being the open unit interval $(0,1)$.  Our arguments and results go through in the same way if a smaller open interval needs to be used, but this is equivalent to our setting by rescaling and so we omit further details.  The major advantage of this class of self-similar measures is that the Bernoulli measures defined on the product system are themselves self-similar measures and, moreover, this forms a 1-parameter family which makes our analysis more transparent. Let $p,q \in (0,1)$ with $p \neq q$ and first consider the set of all couplings of $\mu_p$ and $\mu_q$.  This consists of measures $\gamma$ supported on $F \times F$, such that the projection of $\gamma$ onto the first coordinate is $\mu_p$ and the projection of $\gamma$ onto the second coordinate is $\mu_q$.  Since $F \times F$ is itself a self-similar set, it is natural to consider self-similar measures $\gamma$ with the desired marginals.  Indeed, consider the four similarity maps on $[0,1]^2$ defined by
\[
S_{1,1}(x,y)  \ =  \ (S_1(x), S_1(y))
\]
\[
S_{1,2}(x,y)  \ =  \ (S_1(x), S_2(y))
\]
\[
S_{2,1}(x,y)  \ =  \ (S_2(x), S_1(y))
\]
and
\[
S_{2,2}(x,y)  \ =  \ (S_2(x), S_2(y)).
\]
In fact $F \times F$ is the self-similar set associated to this system and, moreover, the only associated self-similar measures which are also couplings of $\mu_p$ and $\mu_q$ are given by probability vectors of the form
\[
(r, \,  p-r, \,  q-r, \, 1-p-q+r)
\]
associated to the above maps in the given order and where $r$ must be chosen to satisfy
\[
\max\{0,p+q-1\} \ < \  r \  < \  \min \{p,q\}.
\]
We will denote the open interval consisting of such $r$ by $\Lambda_{p,q}$ and, for a given $r \in \Lambda_{p,q}$, the associated self-similar measure will be denoted by $\gamma_r \in \Gamma(\mu_p, \mu_q)$.  The main goal of this paper is to compute 
\begin{equation} \label{keyintegral}
\Bigg( \int_{F \times F} \lvert x - y \rvert^\rho  \, \text{d} \gamma_r(x,y) \Bigg)^{1/\rho}
\end{equation}
for $\rho = 1,2$ (Theorems \ref{gammarint} and \ref{gammarint2}), which has the ancillary benefit of allowing us to estimate the 1st and 2nd Wasserstein distances (Corollaries \ref{main} and \ref{upperbound2}).

\section{Results}

\subsection{First and second moments for our family of self-similar couplings}

Our first main result gives an explicit formula for the 1st moment.

\begin{thm} \label{gammarint}
For $r \in  \Lambda_{p,q}$, we have
\[
 \int_{F \times F} \lvert x - y \rvert  \, \text{d} \gamma_r(x,y) \ = \ \frac{t_2-t_1}{1-c} \, \frac{c(p-q)^2+(1-c)(p+q-2r)}{(1-c)+c(p+q-2r)}.
\]
\end{thm}

We will prove Theorem \ref{gammarint} in Section \ref{gammarintproof}.   For fixed $p,q, t_1, t_2$ and $c$, consider the function $\Phi_1:\mathbb{R} \to \hat{\mathbb{R}}$ defined by
\[
\Phi_1(r) \ = \  \frac{t_2-t_1}{1-c} \,\frac{c(p-q)^2+(1-c)(p+q-2r)}{(1-c)+c(p+q-2r)}.
\]
Elementary analysis yields the following basic facts.

\begin{prop} \label{basicexp}
For fixed $p,q, t_1, t_2, c$ we have that $\Phi_1$ is
\begin{itemize}
\item[(1)]  a rational function of $r$ with one simple pole at
\[
\frac{p+q}{2}+\frac{1-c}{2c}
\]
lying strictly to the right of the interval $\Lambda_{p,q}$.
\item[(2)] has one root at
\[
 \frac{p+q}{2} + \frac{c(p-q)^2}{2(1-c)}
\]
lying strictly to the right of the interval $\Lambda_{p,q}$ and strictly to the left of the pole.
\item[(3)] is differentiable at all $r \in \mathbb{R}$, apart from at the pole, with
\[
\Phi_1'(r) \ = \  2(t_2-t_1)\frac{c^2(p-q)^2-(1-c)^2}{(1-c)\big(1-c +c(p+q-2r)\big)^2} \ < \ 0
\]
and is thus strictly decreasing on the interval $\Lambda_{p,q}$.
\item[(4)] is strictly positive on the interval $\Lambda_{p,q}$.
\end{itemize}
\end{prop}

It is reassuring that the interval we are interested in, namely $\Lambda_{p,q}$, does not contain either the pole or the root of $\Phi_1$, as otherwise we would have a contradiction. Our second main result gives an explicit formula for the second moment.
\begin{thm} \label{gammarint2}
For $r \in  \Lambda_{p,q}$, we have
\[
\Bigg( \int_{F \times F} \lvert x - y \rvert^2  \, \text{d} \gamma_r(x,y) \Bigg)^{1/2} \ = \ \frac{t_2-t_1}{1-c} \sqrt{ \frac{2c(p-q)^2+(1-c)(p+q-2r)}{1+c}}.
\]
\end{thm}
We will prove Theorem \ref{gammarint2} in Section \ref{gammarintproof2}.  For fixed $p,q, t_1, t_2$ and $c$, consider the partial function $\Phi_2$ on $\mathbb{R}$ defined by
\[
\Phi_2(r) \ = \ \frac{t_2-t_1}{1-c} \sqrt{ \frac{2c(p-q)^2+(1-c)(p+q-2r)}{1+c}},
\]
wherever the positive square root exists.  Elementary analysis yields the following basic facts.

\begin{prop} \label{basicexp2}
For fixed $p,q, t_1, t_2, c$ we have that
\begin{itemize}
\item[(1)] $\Phi_2$ is the square root of a linear function.
\item[(2)] $\Phi_2$ has one root at
\[
 \frac{p+q}{2} + \frac{c(p-q)^2}{(1-c)}
\]
lying strictly to the right of the interval $\Lambda_{p,q}$.  Moreover, it is defined to the left of this root and undefined to the right so, in particular, its domain includes $\Lambda_{p,q}$.
\item[(3)] $\Phi_2^2$ is differentiable at all $r \in \mathbb{R}$ with 
\[
(\Phi_2^2)'(r) \ = \ \frac{-2(t_2-t_1)^2}{1-c^2} \ < \ 0
\]
and so $\Phi_2$ is strictly decreasing on the interval $\Lambda_{p,q}$.
\item[(4)] $\Phi_2$ is strictly positive on the interval $\Lambda_{p,q}$.
\end{itemize}
\end{prop}

\begin{figure}[H]
	\centering
	\includegraphics[width=155mm]{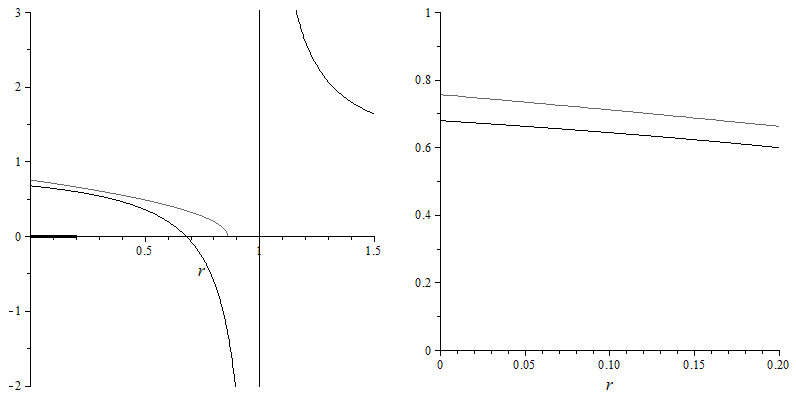}
\caption{A plot of $\Phi_1(r)$ (black) and $\Phi_2(r)$ (grey) for $c=0.5, \, t_1 = 0, \, t_2 = 0.5, \, p= 0.2$ and $q=0.8$, which is representative of the typical shapes.  The interval $\Lambda_{p,q} = (0,0.2)$ is shown in bold on the left and on the right we show $\Phi_1(r)$ and $\Phi_2(r)$ restricted to this interval.}
\end{figure}

\subsection{Applications to Wasserstein distances}

Taking the infima of our formulae for (\ref{keyintegral}) yields natural upper bounds for the 1st and 2nd Wasserstein distances.  It is not, however, \emph{a priori} obvious that the actual values can be realised by self-similar couplings, as there are many other couplings of $\mu_p$ and $\mu_q$ other than the self-similar measures $\gamma_r$.  However, it turns out that this class of measures is enough to realise the 1st Wasserstein distance.  To compute the lower bound, we employ the Kantorovich-Rubinstein duality theorem which gives a useful reformulation of the 1st Wasserstein metric in terms of $1$-Lipschitz functions on $X$.  It states that, for $\mu, \nu \in \mathcal{P}(X)$,
\[
W_1(\mu, \nu) \ = \ \sup \bigg\{ \int_X \phi(x)  \, \text{d}(\mu-\nu) (x) \ : \ \phi: X \to \mathbb{R} \text{ and } \text{Lip}(\phi) \leq 1 \bigg\},
\]
where $\text{Lip}(\phi)$ is the Lipschitz constant of $\phi$.  We will use this to estimate $W_1(\mu_p,\mu_q) $ using $1$-Lipschitz functions $\phi:  F \to \mathbb{R}$ of the form $\phi(x) = \lambda x$ for $\lambda \in [-1,1]$.  
\begin{prop} \label{lowerbound}
For all $p,q, t_1, t_2, c, \lambda$,
\[
 \int_{F } \lambda x \, \text{\emph{d}}(\mu_p-\mu_q)(x)  \ = \ \frac{\lambda (p-q)(t_1  - t_2 ) }{1-c}.
\]
\end{prop}
We will prove Proposition \ref{lowerbound} in Section \ref{lowerboundproof}. Combined with Theorem \ref{gammarint}, this allows the precise computation of the 1st Wasserstein distance.
\begin{cor} \label{main}
For all $p,q, t_1, t_2, c$,
\[
W_1(\mu_p, \mu_q)  \ =  \  \frac{t_2-t_1}{1-c} \, \lvert p-q \rvert.
\]
\end{cor}
\begin{proof}
This follows easily since
\[
\sup_{\lambda \in [-1,1]} \frac{\lambda (p-q)(t_1  - t_2 ) }{1-c}  \ \leq \ W_1(\mu_p, \mu_q)  \ \leq \ \inf_{r \in \Lambda_r} \Phi_1(r) \ = \ \Phi_1 \big(\min\{p,q\} \big)
\]
by the Kantorovich-Rubinstein duality theorem, Theorem \ref{gammarint}, Proposition \ref{basicexp2} (3) and the definition of $W_1$.
\end{proof}
An immediate consequence of Theorem \ref{gammarint} and Corollary \ref{main} is that the 1st Wasserstein distance between $\mu_p$ and $\mu_q$ is realised by the self-similar coupling associated with $r= \min\{p, q\}$.  Although this measure is not strictly in our class because it renders one of the probabilities equal to zero, it can be extracted by weak convergence.  Interestingly, this `realising measure' does not have full support $F \times F$.  We can also derive non-trivial upper and lower bounds for the 2nd Wasserstein distance.

\begin{cor} \label{upperbound2}
For all $p,q, t_1, t_2, c$,
\[
\frac{t_2-t_1}{1-c} \, \lvert p-q \rvert \ \leq \ W_2(\mu_p, \mu_q)  \ \leq \ \frac{t_2-t_1}{1-c} \sqrt{ \frac{2c(p-q)^2+(1-c)\lvert p-q\rvert}{1+c}}.
\]
\end{cor}

\begin{proof}
This follows easily since
\[
W_1(\mu_p, \mu_q) \ \leq \ W_2(\mu_p, \mu_q)  \ \leq \ \inf_{r \in \Lambda_r} \Phi_2(r) \ = \ \Phi_2 \big(\min\{p,q\} \big)
\]
by H\"older's inequality, Proposition \ref{basicexp2} (3) and the definition of $W_2$.
\end{proof}

\begin{rem}
It has recently been drawn to our attention by Tapio Rajala that our formula for $W_1(\mu_p, \mu_q)$, given in Corollary \ref{main}, can also be obtained by a simpler method involving \emph{monotone rearrangement}.  However, this technique also encounters difficulties when there are more than one map, the ambient space has dimension greater than one, or higher moments are considered.  See Section \ref{discussion} for more details. We reiterate that the main interest of this paper is the explicit derivation of the 1st and 2nd moment integrals for the family of self-similar couplings and not the formula for $W_1$.
\end{rem}

\subsection{Further analysis and examples}

In this section we use Theorem \ref{main} to illustrate a strange phenomenon regarding how $W_1(\mu_p, \mu_q)$ depends on the contraction parameter $c$.  Fix $p,q \in (0,1)$ with $p \neq q$ and consider the following two `extremal settings'.  Firstly, if $t_1 = 0$ and $t_2=1-c$, then we are in the setting of the classical middle $(1-2c)$ Cantor set.  In this case, the construction intervals are kept as far apart as possible.  Secondly, if $t_1$ and $t_2$ are chosen such that $t_2-t_1 = c$, then the construction intervals are as close together as possible.  Strikingly, in the first setting the formula for the 1st Wasserstein distance reduces to
\[
W_1(\mu_p, \mu_q) \ = \ \lvert p-q \rvert
\]
which does not depend on $c$.  The physical interpretation of this is that it takes the same amount of `work' to transform a $(p,1-p)$ self-similar measure on the whole interval $(c=1/2)$ into a $(q,1-q)$ self-similar measure on the whole interval as it takes to transform a $(p,1-p)$ self-similar measure on the middle third Cantor set $(c=1/3)$ into a $(q,1-q)$ on the middle third Cantor set.  At first sight this may seem counter intuitive, but the heuristic explanation is as follows.  Our formula shows that $W_1(\mu_p, \mu_q)$ increases in $(t_2-t_1)$ and in $c$.  Choosing $t_2-t_1 = 1-c$ means that we decrease $t_2-t_1$ at the precise rate required to cancel out the effect of increasing $c$. However, it is rather neat that this precise rate puts us exactly in the case of the middle $(1-2c)$ Cantor set, which is in some sense our most natural example.  Interestingly, in this case our upper bound for the 2nd Wasserstein distance reduces to
\[
\sqrt{ \frac{2c(p-q)^2+(1-c)\lvert p-q\rvert}{1+c}}
\]
which retains dependence on $c$.  Indeed, our formula's dependence on $c$ cannot be removed by any choice of $t_1$ and $t_2$ as functions of the other data.

\begin{figure}[H]
	\centering
	\includegraphics[width=160mm]{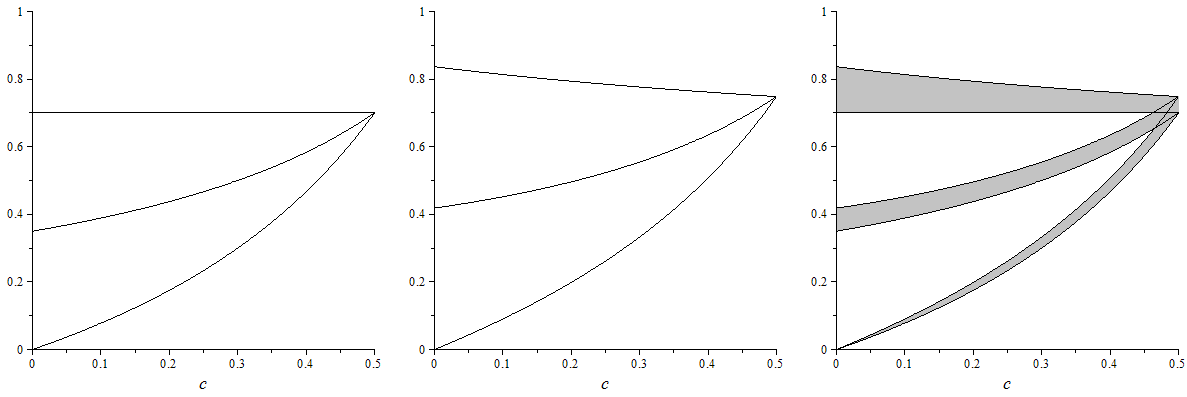}
\caption{Left: A plot of $W_1(\mu_p, \mu_q)$ as a function of $c$ for three different choices of $t_1$ and $t_2$.  In each case, $p=0.2$ and $q = 0.9$.  The maximal graph corresponds to the middle $(1-2c)$ Cantor set case, where there is no $c$ dependence; the minimal graph corresponds to the $t_2-t_1 = c$ case, where the dependence on $c$ is most pronounced; and the middle graph is an intermediate example where $t_1=0$ and $t_2 = 1/2$.  Middle: a plot of the corresponding second moment integral, which gives a non-trivial upper bound on $W_2(\mu_p, \mu_q)$.  Again, the maximal graph corresponds to the middle $(1-2c)$ Cantor set case and the minimal graph corresponds to the $t_2-t_1 = c$ case.  Right: a plot showing the error in the bounds on $W_2(\mu_p, \mu_q)$ given by Corollary \ref{upperbound2} in each of the three cases.}
\end{figure}

Observe that the three measures in question are the same when $c=1/2$ and consequently the three graphs coincide at that point in both cases.  In the $t_2-t_1 = c$ case, $W_1(\mu_p, \mu_q)$ and $W_2(\mu_p, \mu_q)$ converge to 0 as $c$ converges to 0.  This is because if $c$ were chosen to be zero, both $\mu_p$ and $\mu_q$ would be unit point masses at $t_1=t_2$.

\section{Proofs}

\subsection{A useful technical lemma} \label{keyint}

In this section we prove an important technical lemma, which holds the key to proving Theorem \ref{gammarint} and Theorem \ref{gammarint2}.

\begin{lma} \label{techlem}
We have
\[
\int_{F \times F} (x - y)  \, \text{d}\gamma_r(x,y) \ = \ \frac{4(t_2-t_1)(q-r)(p-r)}{(p-q)\big(1-c+c(p+q-2r)\big)} - \bigg( \frac{p+q-2r}{p-q} \bigg) \int_{F \times F} \lvert x - y \rvert  \, \text{d} \gamma_r(x,y) .
\]
\end{lma}
\begin{proof}
 Let  $\mathcal{I} = \{(i,j): i,j \in \{1,2\}\}$ and $ \mathcal{I}^k$ be the set of strings over $\mathcal{I}$ of length $k$.  We write $\textbf{\emph{i}} = (\textbf{\emph{i}}_1, \textbf{\emph{i}}_2, \dots , \textbf{\emph{i}}_k) = ((i_1, j_1), (i_2,j_2), \dots, (i_k, j_k) ) \in \mathcal{I}^k$ and for such $\textbf{\emph{i}} \in \mathcal{I}^k$ write
\[
S_{\textbf{\emph{i}}} = S_{(i_1,j_1)} \circ \cdots \circ S_{(i_k,j_k)}.
\]
Also, for $k = 0, 1,2, \dots$ write
\[
\underline{\mathcal{I}}_{k} \ = \  \big\{ \textbf{\emph{i}} \in \mathcal{I}^{k+1} : (i_1, j_1), \dots, (i_{k}, j_{k}) \in \{(1,1),(2,2)\} \text{ and } (i_{k+1}, j_{k+1}) = (2,1) \big\} ,
\]
\[
\overline{\mathcal{I}}_k \ = \  \big\{ \textbf{\emph{i}} \in \mathcal{I}^{k+1} : (i_1, j_1), \dots, (i_{k}, j_{k}) \in \{(1,1),(2,2)\} \text{ and } (i_{k+1}, j_{k+1}) = (1,2) \big\},
\]
\[
\underline{F \times F}_k \ = \  \bigcup_{\textbf{\emph{i}} \in \underline{\mathcal{I}}_{k}} S_\textbf{\emph{i}}(F \times F),
\]
\[
\overline{F \times F}_k \ = \  \bigcup_{\textbf{\emph{i}} \in \overline{\mathcal{I}}_{k}} S_\textbf{\emph{i}}(F \times F),
\]
\[
\underline{F \times F} \ = \   \bigcup_{k=0}^{\infty}  \underline{F \times F}_k
\]
and
\[
\overline{F \times F} \ = \    \bigcup_{k=0}^{\infty}  \overline{F \times F}_k.
\]
We form this decomposition for the following reason.  We wish to relate the integral of $\lvert x-y \rvert$ over $\underline{F \times F}$ to the integral over $\overline{F \times F}$ but there is not a straightforward way to transform one integral into the other.  However, there is a natural way to transform the integral over $\underline{F \times F}_k$ to the integral over $\overline{F \times F}_k$, which we now demonstrate.

\begin{figure}[H]
	\centering
	\includegraphics[width=162mm]{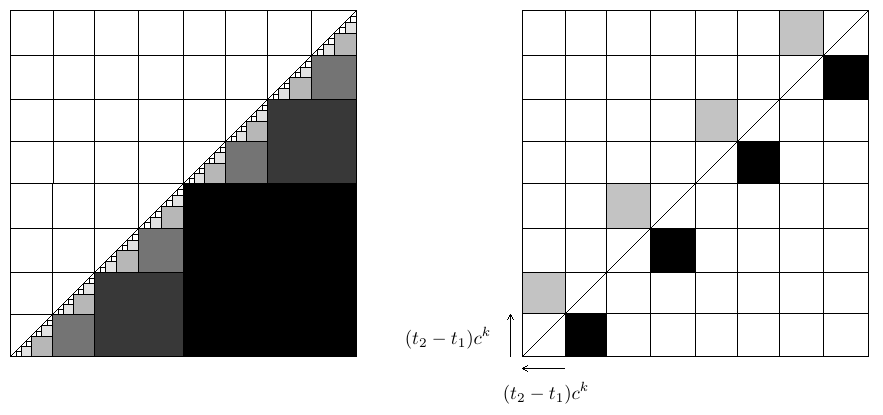}
\caption{Left: a picture indicating the regions in which the $\underline{F \times F}_k$ reside.  The region for $\underline{F \times F}_0$ is shown in black and the regions for $\underline{F \times F}_k$ (which are each the union of $2^k$ squares of side length $c^{k+1}$) then get lighter as $k$ increases.  Right: An indication of how the integral over $\underline{F \times F}_k$ (black) can be transformed into the integral over $\overline{F \times F}_k$ (grey) by shifting the first coordinate left and the second coordinate up, both by $(t_2-t_1)c^k$.  In both pictures $c=1/2$.}
\end{figure}

First note that
\[
\frac{\gamma_r\big(\underline{F \times F}_k\big)}{\gamma_r\big(\overline{F \times F}_k\big)}  \ = \  \frac{(1-p-q+2r)^k(q-r)}{(1-p-q+2r)^k(p-r)}  \ = \  \frac{q-r}{p-r}
\]
for all $k$.  Using this and the substitution indicated in the above figure yields
\begin{eqnarray*}
\int_{\underline{F \times F}_k} \lvert x - y \rvert  \, \text{d}\gamma_r(x,y) &=& \int_{\underline{F \times F}_k} ( x - y )  \, \text{d}\gamma_r(x,y) \\ \\
&=& \frac{q-r}{p-r}   \int_{\overline{F \times F}_k} ( x+(t_2-t_1)c^k )- (y-(t_2-t_1)c^k )  \, \text{d}\gamma_r(x,y) \\ \\
&=& \frac{q-r}{p-r}   \int_{\overline{F \times F}_k} (x-y)  \, \text{d}\gamma_r(x,y) \ + \ \frac{q-r}{p-r}   2(t_2-t_1) c^k \int_{\overline{F \times F}_k}  \text{d}\gamma_r(x,y)  \\ \\
&=& - \frac{q-r}{p-r}   \int_{\overline{F \times F}_k} \lvert x-y \rvert  \, \text{d}\gamma_r(x,y) \ + \ 2 \frac{q-r}{p-r}   (t_2-t_1) c^k \gamma_r\big(\overline{F \times F}_k\big)  \\ \\
&=& 2 \frac{q-r}{p-r}   (t_2-t_1) c^k (1-p-q+2r)^k(p-r)\  -  \ \frac{q-r}{p-r}   \int_{\overline{F \times F}_k} \lvert x-y \rvert  \, \text{d}\gamma_r(x,y)  \\ \\
&=&  2(t_2-t_1)(q-r) \big(c (1-p-q+2r) \big)^k \  -  \ \frac{q-r}{p-r}   \int_{\overline{F \times F}_k} \lvert x-y \rvert  \, \text{d}\gamma_r(x,y).
\end{eqnarray*}
Observe that
\begin{equation} \label{disjointcond0}
\gamma_r\big(\underline{F \times F}_k \cap \underline{F \times F}_{k'}\big) = \gamma_r\big(\overline{F \times F}_k \cap \overline{F \times F}_{k'}\big)   = 0 
\end{equation}
for $k \neq k'$ and so summing over $k$ gives
\begin{eqnarray}
\int_{\underline{F \times F}} \lvert x - y \rvert  \, \text{d}\gamma_r(x,y) &=& \sum_{k=0}^\infty  \int_{\underline{F \times F}_k} \lvert x - y \rvert  \, \text{d}\gamma_r(x,y)  \nonumber \\ \nonumber \\
&=& 2(t_2-t_1)(q-r) \sum_{k=0}^\infty    \big(c (1-p-q+2r) \big)^k \  -  \ \frac{q-r}{p-r} \sum_{k=0}^\infty     \int_{\overline{F \times F}_k} \lvert x-y \rvert  \, \text{d}\gamma_r(x,y) \nonumber \\ \nonumber \\
&=&  \frac{2(t_2-t_1)(q-r)}{1-c (1-p-q+2r)} \  -  \   \frac{q-r}{p-r}   \int_{\overline{F \times F}} \lvert x-y \rvert  \, \text{d}\gamma_r(x,y). \label{est1}
\end{eqnarray}
Applying (\ref{est1}) we get
\begin{eqnarray*}
\int_{F \times F} \lvert x - y \rvert  \, \text{d} \gamma_r(x,y)  & = &   \int_{\underline{F \times F}} \lvert x - y \rvert  \, \text{d}\gamma_r(x,y) \ + \ \int_{\overline{F \times F}} \lvert x - y \rvert  \, \text{d}\gamma_r(x,y) \\ \\
 & = &  \frac{2(t_2-t_1)(q-r)}{1-c (1-p-q+2r)} \ + \ \bigg(1- \frac{q-r}{p-r} \bigg) \int_{\overline{F \times F}} \lvert x - y \rvert  \, \text{d}\gamma_r(x,y)
\end{eqnarray*}
and so
\begin{equation} \label{est2}
\int_{\overline{F \times F}} \lvert x - y \rvert  \, \text{d}\gamma_r(x,y) \ = \  \frac{\int_{F \times F} \lvert x - y \rvert  \, \text{d} \gamma_r(x,y)   -   \frac{2(t_2-t_1)(q-r)}{1-c (1-p-q+2r)}}{1- \frac{q-r}{p-r} }.
\end{equation}

Finally,
\begin{eqnarray*}
\int_{F \times F} (x-y)  \, \text{d}\gamma_r(x,y) &=& \int_{\underline{F \times F}} \lvert x-y \rvert  \, \text{d}\gamma_r(x,y) \ - \ \int_{\overline{F \times F}} \lvert x-y \rvert  \, \text{d}\gamma_r(x,y) \\ \\
&=& \frac{2(t_2-t_1)(q-r)}{1-c (1-p-q+2r)} \ - \ \bigg(1+ \frac{q-r}{p-r} \bigg) \int_{\overline{F \times F}} \lvert x - y \rvert  \, \text{d}\gamma_r(x,y) \qquad \text{by (\ref{est1})} \\ \\
&=& \frac{2(t_2-t_1)(q-r)}{1-c (1-p-q+2r)} \\ \\
&\,&  \  \ - \  \frac{p+q-2r}{p-q}\bigg( \int_{F \times F} \lvert x - y \rvert  \, \text{d} \gamma_r(x,y)   -   \frac{2(t_2-t_1)(q-r)}{1-c (1-p-q+2r)} \bigg)  \qquad \text{by (\ref{est2})} \\ \\
&=& \frac{4(t_2-t_1)(q-r)(p-r)}{(p-q)\big(1-c+c(p+q-2r)\big)} - \bigg( \frac{p+q-2r}{p-q} \bigg) \int_{F \times F} \lvert x - y \rvert  \, \text{d} \gamma_r(x,y) 
\end{eqnarray*}
as required.
\end{proof}

\subsection{Proof of Theorem \ref{gammarint}} \label{gammarintproof}

Fix $r \in  \Lambda_{p,q}$ and write
\[
I_r \ = \  \int_{F \times F} \lvert x - y \rvert  \, \text{d} \gamma_r(x,y) .
\]
It is straightforward to see that
\begin{equation} \label{disjointcond}
\gamma_r \big( S_{i,j}(F \times F) \cap S_{i',j'}(F \times F) \big) = 0
\end{equation}
for $(i,j) \neq (i', j')$ and this gives
\begin{eqnarray*}
I_r &=& r \int_{S_{1,1}(F \times F)} \lvert x - y \rvert  \, \text{d} \big(\gamma_r \circ S_{1,1}^{-1}\big)(x,y) \ + \  (p-r) \int_{S_{1,2}(F \times F)} \lvert x - y \rvert  \, \text{d} \big(\gamma_r \circ S_{1,2}^{-1}\big)(x,y) \\ \\
&\,& + \ (q-r) \int_{S_{2,1}(F \times F)} \lvert x - y \rvert  \, \text{d} \big(\gamma_r \circ S_{2,1}^{-1}\big)(x,y) \ + \ (1-p-q+r) \int_{S_{2,2}(F \times F)} \lvert x - y \rvert  \, \text{d} \big(\gamma_r \circ S_{2,2}^{-1}\big)(x,y)
\end{eqnarray*}
We will treat each of these integrals separately, with the aim of using self-similarity to relate them back to $I_r$. We have
\begin{eqnarray*}
 \int_{S_{1,1}(F \times F)} \lvert x - y \rvert  \, \text{d} \big(\gamma_r \circ S_{1,1}^{-1}\big)(x,y) & = & \int_{S_{1,1}(F \times F)} \lvert c S_1^{-1}(x)+t_1 - cS_1^{-1}(y)-t_1 \rvert  \, \text{d} \big(\gamma_r \circ S_{1,1}^{-1}\big)(x,y) \\ \\
& = &  \int_{F \times F} \lvert cx - cy \rvert  \, \text{d}\gamma_r(x,y) \\ \\
&=& c I_r
\end{eqnarray*}
and, similarly,
\[
 \int_{S_{2,2}(F \times F)} \lvert x - y \rvert  \, \text{d} \big(\gamma_r \circ S_{2,2}^{-1}\big)(x,y)  \ = \  c I_r.
\]
Since for $(x,y) \in S_{1,2}(F \times F)$ we know $x\leq y$, we have
\begin{eqnarray*}
 \int_{S_{1,2}(F \times F)} \lvert x - y \rvert  \, \text{d} \big(\gamma_r \circ S_{1,2}^{-1}\big)(x,y) & = & \int_{S_{1,2}(F \times F)} \Big( c S_2^{-1}(y)+t_2 - cS_1^{-1}(x)-t_1  \Big) \, \text{d} \big(\gamma_r \circ S_{1,2}^{-1}\big)(x,y) \\ \\
& = &  \int_{F \times F} ( cy - cx )  \, \text{d}\gamma_r(x,y)  \  + \  \int_{F \times F} (t_2-t_1)  \, \text{d}\gamma_r(x,y) \\ \\
& = & - c \int_{F \times F} (x - y)  \, \text{d}\gamma_r(x,y)  \  + \   (t_2-t_1).
\end{eqnarray*}
It can be shown similarly that
\[
 \int_{S_{2,1}(F \times F)} \lvert x - y \rvert  \, \text{d} \big(\gamma_r \circ S_{2,1}^{-1}\big)(x,y) \ = \   c \int_{F \times F} (x - y)  \, \text{d}\gamma_r(x,y)  \  + \   (t_2-t_1).
\]

Putting the above integral estimates together and using Lemma \ref{techlem} gives
\begin{eqnarray*}
I_r & =& rcI_r \\ \\
&\,&  -  \ \frac{4c(t_2-t_1)(q-r)(p-r)^2}{(p-q)\big(1-c+c(p+q-2r)\big)} + \frac{c(p-r)(p+q-2r)}{p-q}I_r \ + \ (p-r)(t_2-t_1)\\ \\
&\,&  +  \  \frac{4c(t_2-t_1)(q-r)^2(p-r)}{(p-q)\big(1-c+c(p+q-2r)\big)} -  \frac{c(q-r)(p+q-2r)}{p-q}  I_r \ + \ (q-r)(t_2-t_1)\\ \\
&\,&  +  \ (1-p-q+r)cI_r \\ \\
&=& c I_r \ + \  (t_2-t_1)\bigg(p+q-2r - \frac{4c(q-r)(p-r)}{1-c+c(p+q-2r)} \bigg)
\end{eqnarray*}
which upon solving for $I_r$ yields
\[
I_r \ = \ (t_2-t_1)\frac{c(p-q)^2+(1-c)(p+q-2r)}{(1-c)^2+c(1-c)(p+q-2r)}
\]
as required. If at any point during the above proof the reader was concerned that we divided by zero, recall that the above analysis was only valid for $r$ in the \emph{open} interval $\Lambda_{p,q}$, which precludes this eventuality. 
\hfill \qed

\newpage

\subsection{Proof of Theorem \ref{gammarint2}} \label{gammarintproof2}

Fix $r \in  \Lambda_{p,q}$ and write
\[
I_{r,2} \ = \  \int_{F \times F} \lvert x - y \rvert^2  \, \text{d} \gamma_r(x,y) .
\]
Similar to the proof of Theorem \ref{gammarint2}, and using (\ref{disjointcond}), we have
\begin{eqnarray*}
I_{r,2} &=& r \int_{S_{1,1}(F \times F)} \lvert x - y \rvert^2  \, \text{d} \big(\gamma_r \circ S_{1,1}^{-1}\big)(x,y) \ + \  (p-r) \int_{S_{1,2}(F \times F)} \lvert x - y \rvert^2 \, \text{d} \big(\gamma_r \circ S_{1,2}^{-1}\big)(x,y) \\ \\
&\,& + \ (q-r) \int_{S_{2,1}(F \times F)} \lvert x - y \rvert^2  \, \text{d} \big(\gamma_r \circ S_{2,1}^{-1}\big)(x,y) \ + \ (1-p-q+r) \int_{S_{2,2}(F \times F)} \lvert x - y \rvert^2  \, \text{d} \big(\gamma_r \circ S_{2,2}^{-1}\big)(x,y)
\end{eqnarray*}
As before, we will treat each of these integrals separately, with the aim of using self-similarity to relate them back to $I_{r,2}$. We have
\begin{eqnarray*}
 \int_{S_{1,1}(F \times F)} \lvert x - y \rvert^2  \, \text{d} \big(\gamma_r \circ S_{1,1}^{-1}\big)(x,y) & = & \int_{S_{1,1}(F \times F)} \lvert c S_1^{-1}(x)+t_1 - cS_1^{-1}(y)-t_1 \rvert^2  \, \text{d} \big(\gamma_r \circ S_{1,1}^{-1}\big)(x,y) \\ \\
& = &  \int_{F \times F} \lvert cx - cy \rvert^2  \, \text{d}\gamma_r(x,y) \\ \\
&=& c^2 I_{r,2}
\end{eqnarray*}
and, similarly,
\[
 \int_{S_{2,2}(F \times F)} \lvert x - y \rvert^2  \, \text{d} \big(\gamma_r \circ S_{2,2}^{-1}\big)(x,y)  \ = \  c^2 I_{r,2}.
\]
Also, we have
\begin{eqnarray*}
 \int_{S_{1,2}(F \times F)} \lvert x - y \rvert ^2 \, \text{d} \big(\gamma_r \circ S_{1,2}^{-1}\big)(x,y) & = & \int_{S_{1,2}(F \times F)} \Big( c S_2^{-1}(y)+t_2 - cS_1^{-1}(x)-t_1  \Big)^2 \, \text{d} \big(\gamma_r \circ S_{1,2}^{-1}\big)(x,y) \\ \\
& = & \int_{F \times F} \Big( c(y-x)+(t_2 -t_1)  \Big)^2 \, \text{d} \big(\gamma_r \circ S_{1,2}^{-1}\big)(x,y) \\ \\
& = &  c^2 \int_{F \times F} ( y-x )^2  \, \text{d}\gamma_r(x,y)  \  + \   2c (t_2-t_1) \int_{F \times F} ( y-x )  \, \text{d}\gamma_r(x,y)  \\ \\
&\,& \qquad \qquad \qquad  \  + \  \int_{F \times F} (t_2-t_1)^2  \, \text{d}\gamma_r(x,y) \\ \\
& = &  c^2 I_{r,2} \ - \   2c (t_2-t_1) \int_{F \times F} (x-y )  \, \text{d}\gamma_r(x,y)  \ + \  (t_2-t_1)^2.
\end{eqnarray*}
It can be shown similarly that
\[
 \int_{S_{2,1}(F \times F)} \lvert x - y \rvert^2  \, \text{d} \big(\gamma_r \circ S_{2,1}^{-1}\big)(x,y) \ = \   c^2 I_{r,2} \ + \   2c (t_2-t_1) \int_{F \times F} ( x-y )  \, \text{d}\gamma_r(x,y)  \ + \  (t_2-t_1)^2.
\]
Putting the above integral estimates together and using Lemma \ref{techlem} and Theorem \ref{gammarint} gives
\begin{eqnarray*}
I_{r,2} & =& rc^2I_{r,2} \\ \\
&\,&  +  \   c^2(p-r) I_{r,2} \ - \   2c(p-r) (t_2-t_1) \int_{F \times F} (x-y )  \, \text{d}\gamma_r(x,y)  \ + \  (p-r)(t_2-t_1)^2\\ \\
&\,&  +  \  c^2(q-r) I_{r,2} \ + \   2c(q-r) (t_2-t_1) \int_{F \times F} (x-y )  \, \text{d}\gamma_r(x,y)  \ + \  (q-r)(t_2-t_1)^2\\ \\
&\,&  +  \ (1-p-q+r)c^2I_{r,2} \\ \\
&=& c^2 I_{r,2} \ + \  (p+q-2r)(t_2-t_1)^2  \\ \\
&\,&   - \ 2c(p-q)(t_2-t_1) \Bigg(\frac{4(t_2-t_1)(q-r)(p-r)}{(p-q)\big(1-c+c(p+q-2r)\big)} - \bigg( \frac{p+q-2r}{p-q} \bigg) \int_{F \times F} \lvert x - y \rvert  \, \text{d} \gamma_r(x,y) \Bigg) \\ \\
&=& c^2 I_{r,2} \ + \  (p+q-2r)(t_2-t_1)^2  \\ \\
&\,&   - \ 2c(t_2-t_1) \Bigg(\frac{4(t_2-t_1)(q-r)(p-r)}{1-c+c(p+q-2r)} - ( p+q-2r)  \frac{t_2-t_1}{1-c} \, \frac{c(p-q)^2+(1-c)(p+q-2r)}{(1-c)+c(p+q-2r)} \Bigg)
\end{eqnarray*}
which upon solving for $I_{r,2}$ and taking square roots yields
\begin{eqnarray*}
\Bigg( \int_{F \times F} \lvert x - y \rvert^2  \, \text{d} \gamma_r(x,y) \Bigg)^{1/2}  \ = \ I_{r,2}^{1/2}  &=&  \Bigg((t_2-t_1)^2\frac{2c(p-q)^2+(1-c)(p+q-2r)}{(1-c)^2(1+c)} \Bigg)^{1/2} \\ \\
&=& \frac{t_2-t_1}{1-c} \sqrt{ \frac{2c(p-q)^2+(1-c)(p+q-2r)}{1+c}}
\end{eqnarray*}
as required. Again, $r$ in the \emph{open} interval $\Lambda_{p,q}$ precludes the eventuality that we divided by zero at any point in the above proof.   Also, we take the positive square root in the final two lines, which we know exists by Proposition \ref{basicexp2}.
\hfill \qed

\subsection{Proof of Proposition \ref{lowerbound}} \label{lowerboundproof}

Let $\lambda \in [-1,1]$.  We have
\begin{eqnarray*}
 \int_{F } \lambda x \, \text{d}\mu_p(x) &=&   p \int_{S_1(F) } \lambda x \, \text{d} \big(\mu_p \circ S_1^{-1}\big) (x)  \ + \   (1-p) \int_{S_2(F) } \lambda x \, \text{d} \big(\mu_p \circ S_2^{-1}\big) (x) \\ \\
 &=&   p \int_{S_1(F) } \lambda \big( c(S_1^{-1}(x)) + t_1 \big)  \, \text{d} \big(\mu_p \circ S_1^{-1}\big) (x) \\ \\
&\,& \qquad  \qquad  \qquad  \qquad  \qquad  \qquad \ + \   (1-p) \int_{S_2(F) }  \lambda \big( c(S_2^{-1}(x)) + t_2 \big) \, \text{d} \big(\mu_p \circ S_2^{-1}\big) (x) \\ \\
 &=&   pc \int_{F} \lambda  x  \, \text{d} \mu_p (x) \ + \ \lambda p t_1  \ + \  (1-p)c \int_{F} \lambda  x  \, \text{d} \mu_p (x) \ + \ \lambda (1-p) t_2  \\ \\
 &=&   c \int_{F} \lambda  x  \, \text{d} \mu_p (x) \ + \ \lambda \big(p t_1   +  (1-p) t_2 \big)
\end{eqnarray*}
and hence
\[
 \int_{F } \lambda x \, \text{d}\mu_p(x) \  = \  \frac{\lambda \big(p t_1   +  (1-p) t_2 \big) }{1-c}.
\]
This yields
\[
 \int_{F } \lambda x \, \text{d}(\mu_p-\mu_q)(x)  \ = \  \frac{\lambda \big(p t_1   +  (1-p) t_2 \big) }{1-c} \ - \  \frac{\lambda \big(q t_1   +  (1-q) t_2 \big) }{1-c} \ =  \  \frac{\lambda (p-q)(t_1  - t_2 ) }{1-c},
\]
as required. \hfill \qed

\newpage

\section{Discussion of assumptions and future work} \label{discussion}

In this section we discuss possible directions for future work.  We arrange the different aspects roughly in order of what is in our opinion most interesting/most difficult to least interesting/least difficult.
\\ \\
\textbf{Overlaps:} It would be of great interest to allow the construction to have non-trivial overlaps, however, we believe this would be very difficult.  Our IFSs all satisfy the open set condition, which means that the $\mu_p$ and $\gamma_r$ measure of overlaps is zero.  Without this property our argument breaks down at (\ref{disjointcond0}) and (\ref{disjointcond}).  In particular, we would be interested in continuing the graph shown in Figure 2 in the range $c \in (1/2,1)$ in the case where $t_1= 0$ and $t_2 = 1-c$.  The self-similar measures in this setting are known as Bernoulli convolutions and are notoriously difficult to study.  See the paper \cite{bernoulli} for a survey of the rich and interesting history of these measures up until the year 2000.  We conjecture that the graph of $W_1(\mu_p, \mu_q)$ does not remain constant for $c>1/2$, but drops below $\lvert p-q \rvert$ for at least some values of $c$.  It would also be interesting to know if it is continuous.
\\ \\
\textbf{Different contraction ratios:}  It was crucial to our arguments that the contraction ratios of both maps were the same.  The key reason why this is important is that the product set $F \times F$ is itself a self-similar set.  If the contraction ratios were different, then $F \times F$ is strictly self-affine and the self-affine couplings we would look to use are much more difficult to analyse.   We could still decompose the integral in question into four constituent parts corresponding to the four maps in the product system, but the maps $S_{1,2}$ and $S_{2,1}$ would scale by different amounts in different direction.  We do not believe this situation is hopeless, but would perhaps require a different approach.
\\ \\
\textbf{Higher and non-integer moments:}  The same approach will generally work for computing $(\ref{keyintegral})$ for larger $\rho \in \mathbb{N}$.  The key technical difference is that one would require a `higher moments' version of Lemma \ref{techlem}.  It is convenient that the given `1st moment' version of this result is enough to compute the 2nd moment integral in Theorem \ref{gammarint2}.  This is due to the analogue of  Lemma \ref{techlem} being trivial for even powers.  Computing $(\ref{keyintegral})$ for $\rho \in (1, \infty) \setminus \mathbb{N}$ would require a new approach, although many of the same ideas would apply.
\\ \\
\textbf{Higher dimensions:}  We restricted ourselves to self-similar measures supported in $[0,1]$.  It would be equally natural to pose the same question in $[0,1]^n$ or indeed in a general compact metric space.  Provided the contraction ratios are constant, the strategy would proceed as before by looking at self-similar couplings supported on the self-similar product space.  The notation would be more cumbersome, but we see no initial difficulty in making this extension. 
\\ \\
\textbf{More than two maps:}  Our arguments concerning the upper bound go through in a very similar way for self-similar subsets of $[0,1]$ where the contraction ratios are constant, the translations are chosen to guarantee `measure separation' and there are $N$ maps.  This is perhaps the simplest generalisation of our setting -  where we assumed $N=2$.  The main difference is that the space of self-similar couplings is no longer a 1-parameter family, but a $(N-1)^2$-parameter family, and so the corresponding function $\Phi$ would be a map from an open simplex in $\mathbb{R}^{(N-1)^2}$ (analogous to $\Lambda_{p,q}$) to $\mathbb{R}$.  Since one of the primary objectives of this paper is `clarity over generality', we chose not to pursue this setting.
\\ \\
\textbf{Extension of the lower bound:}  Since the form of the functions $\phi$ used in the proof of  Proposition \ref{lowerbound} are so simple, the argument can be easily extended to give a lower bound for the 1st Wasserstein distance in a more general setting, which we describe here for completeness.  Suppose $S_1, \dots, S_N$ are similitudes mapping $[0,1]$ into itself of the form
\[
S_i(x) = c_ix +t_i
\]
where the contraction ratios $c_i \in (0,1)$ and translations $t_i$ have been chosen such that the open set condition is satisfied, i.e. images of the unit interval under two different maps from the IFS are either disjoint or intersect at a single point. Let $\mu_\textbf{\emph{p}}$ and $\mu_\textbf{\emph{q}}$ be the self-similar measures associated to two different probability vectors $\textbf{\emph{p}} = (p_1, \dots, p_N)$ and $\textbf{\emph{q}} = (q_1, \dots, q_N)$ respectively.   Again using test functions of the form $\lambda x$, following the argument in the proof of Proposition \ref{lowerbound} yields
\[
W_1(\mu_\textbf{\emph{p}}, \mu_\textbf{\emph{q}}) \ \geq \ \frac{ \bigg\lvert       \Big(\sum_{i=1}^N p_i t_i \Big) \Big(1-\sum_{i=1}^N q_i c_i \Big) \ - \    \Big(\sum_{i=1}^N q_i t_i \Big) \Big(1-\sum_{i=1}^N p_i c_i \Big)    \bigg\rvert}{\Big(1-\sum_{i=1}^N p_i c_i \Big)\Big(1-\sum_{i=1}^N q_i c_i \Big)},
\]
which if the contraction ratios are all equal to some constant $c \in (0,1)$, simplifies to
\[
W_1(\mu_\textbf{\emph{p}}, \mu_\textbf{\emph{q}}) \ \geq \ \frac{ \Big\lvert \sum_{i=1}^{N} (p_i-q_i)t_i \Big\rvert}{1-c}.
\]
At this point it would seem reasonable to conjecture that these are the correct values for $W_1(\mu_\textbf{\emph{p}}, \mu_\textbf{\emph{q}})$ in both cases, however, this is not true in either case.  Consider the second, less general, situation where the contraction ratios are equal and let $N=3$.  Choose $t_1 = 0$, $t_2 = c$ and $t_3=2c$.  Then for the probability vectors $\textbf{\emph{p}} = (0.4,0.5,0.1)$ and $\textbf{\emph{q}} = (0.6,0.1,0.3)$, the above lower bound evaluates to
\[
 \frac{ \Big\lvert (0.5-0.1)c+(0.1-0.3)\cdot 2c\Big\rvert}{1-c} \ = \  0,
\]
despite the measures being different.  This shows that if the Kantorovich-Rubinstein duality theorem is to be used to derive the lower bound in this more general setting, then more complicated test functions will have to be used.  The above formula does, however, give a promising lower bound in the case where there are two maps, but different contraction ratios.  In particular, 
\[
W_1(\mu_p, \mu_q) \ \geq \ \Bigg\lvert   \frac{    \Big(p t_1 + (1-p) t_2\Big) \Big(1-qc_1 - (1-q)c_2\Big) \ - \    \Big(qt_1 + (1-q)t_2 \Big) \Big(1-pc_1 - (1-p)c_2 \Big)   }{\Big(1-pc_1 - (1-p)c_2 \Big)  \Big(1-qc_1 - (1-q)c_2 \Big) }  \Bigg\rvert.
\]
We conjecture that this is the correct value.

\vspace{10mm}

\begin{centering}

\textbf{Acknowledgements}

The author was supported by the EPSRC grant EP/J013560/1.  He thanks Tapio Rajala for helpful discussions of this work.

\end{centering}

\end{document}